\documentclass[a4paper,10pt]{article}
\usepackage[utf8]{inputenc}
\usepackage{vmargin,hhline,pifont}
\usepackage{enumerate,fancyhdr}
\usepackage{amsmath, amsthm, amsfonts, amssymb}
\usepackage{marvosym}

\newtheorem{theorem}{Theorem}[section]
\newtheorem{proposition}[theorem]{Proposition}
\newtheorem{lemma}[theorem]{Lemma}

\theoremstyle{definition}
\newtheorem{definition}[theorem]{Definition}
\newtheorem{example}[theorem]{Example}

\theoremstyle{remark}
\newtheorem{remark}[theorem]{Remark}

\DeclareMathOperator{\Lexp}{Lexp}

\begin{document}

\title{Kac-Moody symmetric spaces of Euclidean type}
\author{Walter Freyn}

\maketitle

\begin{abstract}
We investigate in detail the class of Euclidean affine Kac-Moody symmetric spaces and their orthogonal symmetric affine Kac-Moody algebras (OSAKAs). These spaces are the only class of Kac-Moody symmetric spaces, that is not directly derived from affine Kac-Moody algebras in the classical sense.   
\end{abstract}

\section{Introduction}

\emph{Riemannian symmetric spaces} were introduced by \'Elie Cartan around 1925 as a special class of Riemannian manifolds, defined by the property, that their curvature tensor is parallel. They are thus a generalization of spaces of \emph{constant curvature} such as the sphere, hyperbolic space or Euclidean space. It turns out, that this new curvature condition has important consequences. For example symmetric spaces are homogeneous spaces and as such can be characterized by their huge isometry groups. This observations leads to another characterization and definition of symmetric spaces commonly used today: A \emph{Riemannian symmetric space} is a Riemannian manifold $M$, such that for every point $p\in M$ there is an involutive geodesic symmetric $\rho_p$, such that $p$ is an isolated fixed point of $\rho_p$. Riemannian symmetric spaces can be classified by classifying their isometry groups. 

The classification distinguishes three basic types of finite dimensional Riemannian symmetric spaces: spaces of so-called {\em compact type} having \emph{non-negative  sectional curvature} $K\geq 0$, spaces with \emph{non-positive sectional curvature} $K\leq 0$, called of \emph{non-compact type} and those with vanishing sectional curvature $K\equiv 0$, called of
\emph{Euclidean type}~\cite{Helgason01}. Riemannian symmetric spaces of Euclidean type are isometric to Euclidean space, that is the real space $\mathbb{R}^n$ equipped with its natural Euclidean metric. 

\emph{Kac-Moody algebras} were introduced and first studied in the 60's independently by V.\ G.\ Kac~\cite{Kac68}, R.\ V.\ Moody~\cite{Moody69}, I.\ L.\ Kantor~\cite{Kantor68} and D.-N.\ Verma (unpublished) as a generalization of semisimple Lie algebras. The subclass of \emph{affine Kac-Moody algebras} allows an explicit construction as extensions of loop algebras. For a simple complex Lie algebra $\mathfrak{g}$ the loop algebra $L(\mathfrak{g})$ is defined to consist of all holomorphic maps $f:\mathbb{C}^*\rightarrow \mathfrak{g}$; the associated (affine, non-twisted) Kac-Moody algebra $\widehat{L}(\mathfrak{g})$ is defined as a certain $2$-dimensional extension of $L(\mathfrak{g})$ by a central element $c$ and a derivation $d$. Kac-Moody algebras share most of their algebraic structure elements with finite dimensional simple Lie algebras; consequently it was conjectured, that they should also share their geometry; most importantly the existence of symmetric spaces, associated to affine Kac-Moody algebras was conjectured by C.-L.~Terng in~\cite{Terng95}.

 We introduced and constructed \emph{affine Kac-Moody symmetric spaces} in the authors thesis~\cite{Freyn09} and present the details in a series of several papers~\cite{Freyn07, Freyn12d, Freyn10a, Freyn10d, Freyn12b, Freyn12c, Freyn12e}.  Kac-Moody symmetric spaces are \emph{tame Fr\'echet manifolds}; their metric is of \emph{Lorentzian signature}. They are the fundamental global objects of affine Kac-Moody geometry~\cite{Freyn10a, Heintze06}. Being the closest infinite dimensional counterpart to finite dimensional Riemannian symmetric spaces, they share important aspects of their structure theory.  For example their classification is similar to the one of finite dimensional Riemannian symmetric spaces:
As in this classical situation, there are three distinguished subclasses:

\begin{itemize}
\item affine Kac-Moody symmetric spaces of the {\em compact type},
\item affine Kac-Moody symmetric spaces of the {\em non-compact type},
\item affine Kac-Moody symmetric spaces of the {\em Euclidean type}.
\end{itemize}

In this paper we investigate the interesting special case of \emph{affine Kac-Moody symmetric spaces of Euclidean type}. 
This special case lies in-between the two large classes of Kac-Moody symmetric spaces of compact type and of non-compact type.

While symmetric spaces of the compact type and of the non-compact type are intimately linked to semisimple Lie groups, the class of Euclidean symmetric space behaves differently. In the finite dimensional situation, a Riemannian symmetric space of Euclidean type is isometric to $\mathbb{R}^n$, equipped with its unique scalar product. The isometry group of $\mathbb{R}^n$ is the Euclidean group $E(n):=O(n)\ltimes \mathbb{R}^n$, hence a semidirect product of two subgroups. The important {\em duality} between simply-connected symmetric spaces of compact type and of non-compact type breaks down in this situation: Euclidean symmetric spaces are ``self-dual''.

For affine Kac-Moody symmetric spaces of Euclidean type, the situation is different. The important new feature is the existence of two types of Kac-Moody symmetric spaces of Euclidean type, which are \emph{dual} to each other, one being related to (semisimple) affine Kac-Moody symmetric spaces of the compact type and one to (semisimple) affine Kac-Moody symmetric spaces of the non-compact type;  hence for affine Kac-Moody symmetric spaces, the duality relation extends to the Euclidean case. This new feature can ultimately be seen as a consequence of the Lorentzian signature of the metric.   

Let us describe the content of this article in more detail. In section~\ref{section:tame_Frechet} we review the functional analytic basics of the construction, giving the central definitions of Fr\'echet spaces in subsection~\ref{subsection:Frechet_spaces} and of tame Fr\'echet spaces in subsection~\ref{subsection:tame_Frechet_spaces}. Furthermore we review the differential geometry of tame Fr\'echet manifolds in subsection~\ref{subsection:Differential_geometry}. Then in section~\ref{section:Euclidean_OSAKAs} we study Euclidean \emph{O}rthogonal \emph{S}ymmetric \emph{K}ac-Moody \emph{A}lgebras (OSAKAs) and describe their relation to Heisenberg algebras. The final section~\ref{section:Symmetric_spaces_Euclidean_type} is then devoted to the construction of the Kac-Moody symmetric spaces of the Euclidean type themselves.

This work is intended as a detailed answer to questions asked to the author after talks concerning this special case; the author is grateful for this interest.

%%%%%%%%%%%%%%%%%%%%%%%%%%%%%%%%%%%%%%%%%%%%%%%%%%%%%%%%%%%%%%%%%%%%%%%%%%%%%%%%%%%%%%%%%%%%%%%%%%%%%%%%%%%%%%%%%%%%%%%%%%%%%%%%%%%%%%%%%%%%%%%%%%%%%%%%%%%%%%%%%%%%%%%%%%%%%%%%%%%%%%%%%%%%%%%%%%%%%%%%%%%%%%%%%%%%%%%%%%%%%%%%%%%%%%%%%%%%%%%%%%%%%%%%%%%%%%%%%%%%%%%%%%%%%%%%%%%%%%%%%%%%%%%%%%%%%%%%%%%%%%%%%%%%%%%%%%%%%%%%%%%%%%%%%%%%%%%%%%%%%%%%%%%%%%%%%%%%%%%%%%%%%%%%%%%%%%%%%%%%%%%%%%%%%%%%%%%%%%%%%%%%%%%%%%%%%%%%%%%%%%%%%%%%%%%%%%%%%%%%%%%%%%%%%%%%%%%%%%%%%%%%%%%%%%%%%%%%%%%%%%%%%%%%%%%%%%%%%%%%%%%%%%%%%%%%%%%%%%%%%%%%%%%%%%%%%%%%%%%%%%%%%%%%%%%%%%%%%%%%%%%%%%%%%%%%%%%%%%%%%%%%%%%%%%%%%%%%%%%%%%%%%%%%%%%%%%%%%%%%%%%%%%%%%%%%%%%%%%%%%%%%%%%%%%%%%%%%%%%%%%%%%%%%%%%%%%%%%%%%%%%%%%%%%%%%%%%%%%%%%%%%%%%%%%%%%%%%%%%%%%%%%%%%%%%%%%%%%%%%%%%%%%%%%%%%%%%%%%%%%%%%%%%%%%%%%%%%%%%%%%%%%%%%%%%%%%%%%%%%%%%%%%%%%%%%%%%%%%%%%%%%%%%%%%%%%%%%%%%%%%%%%%%%%%%%%%%%%%%%%%%%%%%%%%%%%%%%%%%%%%%%%%%%%%%%%%%%%%%%%%%%%%%%%%%%%%%%%%%%%%%%%%%%%%%%%%%%%%%%%%%%%%%%%%%%%%%%%%%%%%%%%%%%%%%%%%%%%%%%%%%%%%%%%%%%%%%%%%%%%%%%%%%%%%%%%%%%%%%%%%%%%%

\section{Tame Fr\'echet manifolds}
\label{section:tame_Frechet}

\subsection{Fr\'echet spaces}
\label{subsection:Frechet_spaces}

In this section we review some standard results about Fr\'echet spaces, Fr\'echet manifolds and Fr\'echet Lie groups. Additional details and proofs can be found in~\cite{Hamilton82, Freyn12e}.

\begin{definition}
A {\em Fr\'echet vector space} is a locally convex topological vector space which is complete, Hausdorff and metrizable. 
\end{definition}

For example Banach spaces are Fr\'echet spaces (the countable collection of norms consists of just one element).
An important example for our theory is the following: Let $\textrm{Hol}(\mathbb{C}, \mathbb{C})$ denote the space of holomorphic functions $f: \mathbb{C} \longrightarrow \mathbb{C}$ and let $B_n(0):=\{z\in \mathbb{C}| |z|\leq e^n\}\subset \mathbb{C}$ denote the closed ball of radius $e^n$.  Let $\|f\|_n:= \displaystyle\sup_{z\in B_n(0)} |f(z)|$. Then $\textrm{Hol}(\mathbb {C}, \mathbb {C}; \|\hspace{3pt}\|_n)$ is a Fr\'echet space.
 
\begin{definition}
 A Fr\'echet manifold is a (possibly infinite dimensional) manifold with charts in a Fr\'echet space such that the chart transition functions are smooth. 
\end{definition}

While it is possible to define Fr\'echet manifolds in this way, the main obstacle in developing a fruitful theory is that there is no general inverse function theorem for smooth maps between Fr\'echet spaces available. Illustrative counterexamples can be found in~\cite{Hamilton82}. This challenge is usually resolved by restricting oneself to a subclass of ``well-behaving'' Fr\'echet spaces. A particularly usefull class is the one of {\em tame Fr\'echet spaces}, introduced in~\cite{Hamilton82}.

\subsection{Tame Fr\'echet spaces}
\label{subsection:tame_Frechet_spaces}

The central problem for all further structure theory of Fr\'echet spaces is a better control of the set of seminorms. For Fr\'echet spaces $F$ and $G$ and a map $\varphi: F\longrightarrow G$ this is done by imposing \emph{tame estimates} similar in spirit to the concept of quasi isometries relating the sequences of norms $\|\varphi(f)\|_n$ on $G$ and $\|f\|_m$ on $F$. Tame Fr\'echet spaces are then defined as complemented subspaces in certain model spaces of Banach-space-valued holomorphic functions.
The main reference of the field is the article~\cite{Hamilton82}. 

The prerequisite for estimating norms under maps between Fr\'echet spaces are estimates of the norms on the Fr\'echet space itself. This is done by a \emph{grading}:  A grading on a Fr\'echet space $F$ is a collection of seminorms $\{\|\hspace{3pt}\|_{n}, n\in \mathbb N_0\}$ that defines the topology and satisfies
$$\|f \|_0\leq \|f\|_1 \leq \|f\|_2 \leq\| f \|_3 \leq \dots \,.$$

\noindent Remark, that all Fr\'echet spaces admits gradings. Let  $(F, \|\hspace{3pt}\|_{n, n\in \mathbb{N}})$ be a Fr\'echet space. Then 
 $\widetilde{\|\hspace{3pt} \|}_{n}:=\displaystyle\sum_{i=1}^n \|\hspace{3pt}\|_i$ defines a grading on $F$. Remark, that the Fr\'echet topologies of $F$ induced by $\|\hspace{3pt}\|_{n}$ and $\widetilde{\|\hspace{3pt} \|}_{n}$ coincide.

\begin{definition}[Tame equivalence of gradings]
Let $F$ be a graded Fr\'echet space, $r,b \in \mathbb{N}$ and $C(n), n\in \mathbb{N}$ a sequence with values in $\mathbb{R}^+$. The two gradings  $\{\|\hspace{3pt}\|_n\}$ and $\{\widetilde{\|\hspace{3pt}\|}\}$ are called $\left(r,b,C(n)\right)$-equivalent iff 
\begin{equation*}
 \|f\|_n \leq C(n) \widetilde{\|f\|}_{n+r} \text{ and }  \widetilde{\|f\|}_n \leq C(n)\|f\|_{n+r} \text{ for all } n\geq b
\,.
\end{equation*}
They are called tame equivalent iff they are $(r,b,C(n))$-equivalent for some values of $(r,b,C(n))$.
\end{definition}

\noindent The following example from~\cite{Hamilton82} is basic:

\begin{example}
Let $B$ be a Banach space with norm $\| \hspace{3pt} \|_B$. Denote by $\Sigma(B)$ the space of all exponentially decreasing sequences $\{f_k\}$, ${k\in \mathbb N_0}$ of elements of $B$.
On this space, we can define different gradings:
\begin{align}
\|f\|_{l_1^n} &:= \sum_{k=0}^{\infty}e^{nk} \|f_k\|_B\\
\|f\|_{l_{\infty}^n}&:= \sup_{k\in \mathbb N_0} e^{nk}\|f_k\|_B
\end{align}
These two gradings $\|f\|_{l_1^n}$ and $\|f\|_{l_{\infty}^n}$ are tame equivalent.
\end{example}

\begin{definition}[Tame linear maps]
Let  $F$, $G$, $G_1$ and $G_2$  denote graded Fr\'echet spaces.
\begin{enumerate}
\item A linear map $\varphi: F\longrightarrow G$ is called $(r,b,C(n))$-tame if it satisfies the inequality
$$\|\varphi(f)\|_n \leq C(n)\|f\|_{n+r}\,.$$
$\varphi$ is called \emph{tame} iff it is $(r,b,C(n))$-tame for some values of $(r,b, C(n))$.
\item A map $\varphi:F\longrightarrow G$ is called a \emph{tame isomorphism} iff it is a linear isomorphism and $\varphi$ and $\varphi^{-1}$ are tame maps.
\item $F$ is a \emph{tame direct summand }of $G$ iff there exist tame linear maps $\varphi: F\longrightarrow G$ and $\psi: G \longrightarrow F$ such that $\psi \circ \varphi: F \longrightarrow F$ is the identity.
\end{enumerate}
\end{definition}

\begin{definition}[Tame Fr\'echet space]
$F$ is a tame Fr\'echet space iff there is a Banach space $B$ such that $F$ is a tame direct summand of $\Sigma(B)$.
\end{definition}

\noindent Essentially, tame Fr\'echet spaces are spaces of Banach-valued holomorphic functions~\cite{Freyn12c}.

\begin{lemma}[Constructions of tame Fr\'echet spaces]~
\label{constructionoftamespaces}
\begin{enumerate}
	\item A tame direct summand of a tame Fr\'echet space is tame.
	\item A Cartesian product of two tame Fr\'echet spaces is tame.
\end{enumerate}
\end{lemma}

\noindent There are various examples of tame Fr\'echet spaces; most importantly the spaces $\textrm{Hol}(\mathbb{C}, \mathbb{C})$ (for details see~\cite{Hamilton82}) and $\textrm{Hol}(\mathbb{C}^*, \mathbb{C})$ (for details see~\cite{Freyn12d}) are tame.

\begin{definition}[Tame Fr\'echet Lie algebra]
\label{tamefrechetLie algebra}
A Fr\'echet Lie algebra $\mathfrak{g}$ is tame iff it is a tame vector space and
$\textrm{ad}(X)$ is a tame linear map for every $X\in \mathfrak{g}$.
\end{definition}

\begin{definition}[Tame Fr\'echet manifold]
A \emph{tame Fr\'echet  manifold} is a Fr\'echet manifold with charts in a tame Fr\'echet space
such that the chart transition functions are smooth tame maps.
\end{definition}

\subsection{Differential geometry of tame Fr\'echet manifolds}
\label{subsection:Differential_geometry}

Let $M$ be a tame Fr\'echet manifold whose charts take values in a Fr\'echet space $F$. Assume furthermore, that for every point $p\in M$ the exponential map is a local diffeomorphism (this is not in general true for tame Fr\'echet manifolds - for counterexamples see~\cite{Freyn12d}). Then the tangent space $T_pM$ can be identified with $F$. Denote by $TM$ the tangent bundle of $M$ and define a vector field on $M$ to be a smooth section of $TM$ (for details see~\cite{Hamilton82, Freyn12e}). 

The \emph{vertical bundle} $T_vP\subset TP$ consists of the vertical vectors, that is the vectors $v\subset TP$ such that $v \subset (d\pi)^{-1}(0,x)$ with $x\in M$.
A connection on $TP$ consists of the assignment of a complementary tame Fr\'echet subbundle $T_hP$ of $TP$ (i.e.: such that $TP=T_vP\oplus T_hP$), that is:

\begin{definition}[Connection]
A \emph{connection} $\Gamma$ on $TP$ is the assignment of a complementary subspace of horizontal vectors, such that in terms of any coordinate chart of $P$ with values in $(U\subset F)\times G$ the subspace of horizontal vectors at $T_fP$ consists of all $(h,k)\in F_1\times G$ such that $k=\Gamma(f)\{g,h\}$, where $\Gamma$ is represented in any local chart by a smooth map $\Gamma: (U\subset F)\times G\times F_1 \longrightarrow G$, which is bilinear in $g$ and $h$.  
\end{definition} 

If $P$ is the tangent bundle of $M$, then $G=F_1$. We call a connection symmetric if $\Gamma$ is symmetric in $\{g,h\}$. 

In the finite dimensional case, in order to define curvature, one would now define a tensor field on $M$. 
As dual spaces of Fr\'echet spaces are (with the trivial exception of Banach spaces) not Fr\'echet spaces, this approach is not possible for Fr\'echet manifolds.  In contrast we are forced to use explicit coordinate dependent descriptions in terms of component functions.

\begin{definition}[curvature]
The \emph{curvature} of a connection on a vector bundle $P$ over $M$ is the trilinear map $R: P\times TM\times TM \longrightarrow P$ such that
$$R(f)\{g,h,k\}:= D\Gamma(f)\{g,h,k\}- D\Gamma(f)\{g,k,h\}-\Gamma(f)\{\Gamma(f)\{g,h\},k\}+\Gamma(f)\{\Gamma(f)\{g,k\},h\}\,.$$
\end{definition}

\noindent Similarly metrics are defined by their coordinate functions:

\begin{definition}[metric]
Let $M$ be a tame Fr\'echet manifold and $TM$ its tangential bundle. A \emph{metric} on $TM$ is a smooth bilinear map: $g:TM\times TM \longrightarrow \mathbb{R}$. A map $g$ is \emph{smooth} if it is smooth in any local chart.
\end{definition}

\noindent Clearly, $M$ is not complete with respect to $g$ (otherwise $M$ would be a Hilbert manifold); so $g$ is only a weak metric.

Following the finite dimensional convention, we define the index of $g$ to be the dimension of the maximal subspace on which $g$ is negative definite.

As usual a connection is  \emph{compatible} with the metric $g$, iff \small
$$\frac{d}{dt} g\left(V,W\right)=g\left(\frac{DV}{dt},W\right)+g\left(V, \frac{DW}{dt}\right)$$\normalsize
for any vector fields $V$ and $W$ along a curve $c:I\longrightarrow M$. 

\begin{definition}[Levi-Civita connection]
Let $(M,g)$ be a tame Fr\'echet manifold. A \emph{Levi-Civita connection} is a symmetric connection which is compatible with the metric.
\end{definition}

\begin{lemma}
If a Levi-Civita connection exists, it is well defined and unique.
\end{lemma}

\noindent For a proof see~\cite{Freyn12e}. In contrast to this result, the existence of a Levi-Civita connection seems not to be clear in general for tame Fr\'echet manifolds. Nevertheless the existence problem is answered affirmatively in the special case of Lie groups by B.\ Popescu in~\cite{Popescu05}:

\begin{theorem}[Existence of Levi-Civita connection on Fr\'echet Lie group]
\label{existanceoflevicivitaconnection}
Any tame Fr\'echet Lie group $G$ admits a unique left invariant connection such that $\nabla_{X}Y=\frac{1}{2}[X,Y]$ for any pair of left invariant vector fields $X$ and $Y$. It is torsion free. If $G$ admits a biinvariant (pseudo-) Riemannian metric, then $\nabla$ is the corresponding Levi-Civita connection.
\end{theorem}

\noindent For the following result see~\cite{Freyn12e}.

\begin{lemma}[curvature of tame Fr\'echet Lie group]
\label{curvatureoffrechetliegroup}
The curvature tensor of a tame Fr\'echet Lie group is given by:
\begin{displaymath}
R(f)\{g,h,k\}=\frac{1}{4}[[g,h],k]\, .
\end{displaymath}
\end{lemma}

\begin{definition}[Sectional curvature]
\label{def:sectional_curvature}
If $|g\wedge h|^2=\langle g, g \rangle \langle h, h\rangle -\langle g,h\rangle ^2 \neq 0 $, we define the sectional curvature, to be 
\begin{displaymath}K_f(g,h)=\frac{\langle R(f)\{g,h,g\},h \rangle}{|g\wedge h|^2}\,.\end{displaymath}
\end{definition}

\noindent We now define tame Fr\'echet symmetric spaces

\begin{definition}[tame Fr\'echet symmetric space]
\label{tamesymmetricspace}
A tame Fr\'echet manifold $M$ with a weak metric equipped with a Levi-Civita connection is called a symmetric space, iff for all $p\in M$ there is an involutive isometry $\rho_p$, such that $p$ is an isolated fixed point of $\rho_p$.
\end{definition}

\noindent Also the concept of geodesic symmetries extends from the finite dimensional blueprint~\cite{Freyn12e}:

\begin{lemma}[geodesic symmetry]
\label{geodesicsymmetry}
Let $M$ be a tame Fr\'echet pseudo-Riemannian symmetric space.
For each $p\in M$ there exists a normal neighborhood $N_p$ of $p$ such that $\rho_{p}$ coincides with the geodesic symmetry on all geodesics through $p$ in $N_{p}$.
\end{lemma}

\noindent Hence, Kac-Moody symmetric spaces admit the following characterization:

\begin{definition}[Kac-Moody symmetric space]
An (affine) Kac-Moody symmetric space $M$ is a tame Fr\'echet Lorentz symmetric space such that its isometry group $I(M)$ contains a transitive subgroup isomorphic to an affine geometric Kac-Moody group $H$, and the intersection of the isotropy group of a point with $H$ is a loop group of compact type (which may be $0$-dimensional).
\end{definition}

\section{Euclidean OSAKAs}
\label{section:Euclidean_OSAKAs}

\subsection{Geometric Kac-Moody algebras}
We follow the geometric approach to Kac-Moody algebras described in~\cite{Heintze09}. For the algebraic side see the books~\cite{Kac90, Carter05}. For a detailed review of the relation between the two constructions see~\cite{Freyn12b}. 

\noindent Let $\mathfrak{g}$ be a finite dimensional complex, reductive Lie algebra. Hence by definition $\mathfrak{g}:=\mathfrak{g}_{s}\oplus \mathfrak{g}_{a}$ is a direct product of a semisimple Lie algebra $\mathfrak{g}_s$ with an Abelian Lie algebra $\mathfrak{g}_{a}$.  We define the loop algebra $L(\mathfrak{g})$ as follows:
\begin{displaymath}
L(\mathfrak{g}):=\{f:\mathbb{C}^*\longrightarrow \mathfrak{g}\hspace{3pt}| f \textrm{ is holomorphic on $\mathbb{C}^*$} \}\label{abstractkacmoodyalgebra}\,.
\end{displaymath}

\noindent This algebra is a tame Fr\'echet Lie algebra with respect to the family of norms $\|f\|_n:=\sup_{z\in A_n}|f(z)|$ where $A_n$ denotes the annulus $A_n:=\{x\in \mathbb{C}^*| e^{-n}\leq |x|\leq e^{n} \}$.

\begin{definition}
The complex geometric affine Kac-Moody algebra associated to $\mathfrak{g}$ is the algebra
$$\widehat{L}(\mathfrak{g}):=L(\mathfrak{g}) \oplus \mathbb{C}c \oplus \mathbb{C}d\,,$$
equipped with the Lie bracket defined by:
\begin{alignat*}{1}
  [d,f(z)]&:=izf'(z)\,,\\
  [c,c]=[d,d]&:=0\,,\\
  [c,d]=[c,f(z)]&:=0\,,\\
  [f,g](z)&:=[f(z),g(z)]_{0} + \omega(f(z),g(z))c\,.
\end{alignat*}

Here $f\in L(\mathfrak{g})$ and  \begin{displaymath}\omega(f,g):=\textrm{Res}(\langle f, g' \rangle)\,.\end{displaymath}
Hence $\omega$ is a certain antisymmetric $2$-form on $L(\mathfrak g)$, satisfying the cocycle condition.
\end{definition}

A Kac-Moody algebra is called of \emph{semisimple type} of $\mathfrak{g}$ is semisimple and of \emph{Euclidean type} if $\mathfrak{g}$ is Abelian. 

Remark, that we only consider non-twisted affine Kac-Moody algebras. This is no restriction as all twisted affine Kac-Moody algebras are of semisimple type and hence do not appear in connection with symmetric spaces of Euclidean type.

\subsection{OSAKAs}
Similar to (simply connected) finite dimensional Riemannian symmetric spaces, that are classified by associating to them some orthogonal symmetric Lie algebras (OSLA) (see \cite{Helgason01}), Kac-Moody symmetric symmetric spaces can be classified using their orthogonal symmetric affine Kac-Moody algebras (OSAKAs) (see~\cite{Freyn12b, Freyn12e}).

\begin{definition}[Orthogonal symmetric Kac-Moody algebra]
An orthogonal symmetric affine Kac-Moody algebra (OSAKA) is a pair $\left(\widehat{L}(\mathfrak{g}), \widehat{L}(\rho)\right)$ such that
\begin{enumerate}
	 \item $\widehat{L}(\mathfrak{g})$ is a real form of an affine geometric Kac-Moody algebra,
	 \item $\widehat{L}(\rho)$ is an involutive automorphism of $\widehat{L}(\mathfrak{g})$,
	 \item If $\mathfrak{g}=\mathfrak{g}_s\oplus \mathfrak{g}_a$ then the restriction $\textrm{Fix}(\widehat{L}(\rho))|_{\mathfrak{g}_s}$ is a compact Kac-Moody algebra or a compact loop group and $\textrm{Fix}(\widehat{L}(\rho))|_{\mathfrak{g}_a}=0$.
\end{enumerate}
\end{definition}

An OSAKA $(\widehat{L}(\mathfrak{g}), \widehat{L}(\rho))$ is called \emph{irreducible} iff it has no non-trivial Kac-Moody subalgebra invariant under $\widehat{L}(\rho)$.

Similar to orthogonal symmetric Lie algebras in the finite dimensional case (see~\cite{Helgason01}), there are OSAKAs of compact type, of noncompact type and of Euclidean type. In this paper we focus uniquely on the Euclidean case:

\begin{definition}
\label{def:Heisenberg}
The infinite dimensional holomorphic Heisenberg algebra $H_{\epsilon}$ is the Lie algebra generated by basis elements $a_n, n\in \mathbb{Z}$ and a central element $c$ such that 
$$[a_{m}, a_{n}]=\delta_{m,-n}\epsilon c\,.$$
Its elements  are all (possibly infinite) linear combinations $r_{c}c+\sum_{n} r_{n} a_n$ such that $\sum_{n>0}(|a_n|+|a_{-n}|)e^{kn}<\infty$ for all $k\in \mathbb{N}$.
If $\epsilon=1$ then $H_{\epsilon}$ is said to be of \emph{non-compact type}, if $\epsilon=i$ then $H_{\epsilon}$ is said to be of \emph{compact type}.
\end{definition}

\noindent Slightly generalizing this definition, we put:

\begin{definition}
\label{def:generalized_Heisenberg}
The infinite dimensional $k$-Heisenberg algebra $H_{k, \epsilon}$  is the Lie algebra generated by basis elements $a_{n,i}, n\in \mathbb{Z}, i=\{1, \dots, k\}$ and a central element $c$ such that 
$$[a_{m,i}, a_{n,j}]=\delta_{i,j}\delta_{m,-n}\epsilon c\,.$$
$k$ is called the dimension of $H_{k, \epsilon}$. Its elements are all (possibly infinite) linear combinations $r_{c}c+\sum_{n,i} r_{n,i} a_{n,i}$ such that for any $i$ we have $\sum_{n>0}(|a_{n,i}|+|a_{-n,i}|)e^{kn}<\infty$ for all $k\in \mathbb{N}$. If $\epsilon=1$ then $H_{k,\epsilon}$ is said to be of \emph{non-compact type}, if $\epsilon=i$ then $H_{\epsilon}$ is said to be of \emph{compact type}.
\end{definition}

\begin{lemma}\
\begin{itemize}
\item The derived algebra of an irreducible Euclidean Kac-Moody algebra $\widehat{L}(\mathbb{C})$ is isomorphic to the Heisenberg algebra $H_{\epsilon}$.
\item The derived algebras of an reducible Euclidean Kac-Moody algebra $\widehat{L}(\mathfrak{a})$  where $\mathfrak{a}\cong \mathbb{C}^k$ is isomorphic to the generalized Heisenberg algebra $H_{k, \epsilon}$, such that $k=\textrm{dim}(\mathfrak{a})$. 
\end{itemize}
\end{lemma}

\noindent Consequently, the associated group $\widetilde{L(G)}$ is a Heisenberg group (see also~\cite{PressleySegal86}, section 9.5).

So let $H_{k, \epsilon, \mathbb{C}}$ be an infinite dimensional complex Heisenberg algebra. We want to study its real forms. 

\begin{remark}
Heisenberg algebras are related to affine Kac-Moody symmetric spaces of Euclidean type. In the finite dimensional blueprint, their role is taken by the Euclidean Lie algebras $E(k)$. Remark that we do not see the infinite dimensional counterpart of the orthogonal group, but only the translation group. This is partly due to the fact that the associated symmetric space is not of semisimple type. More precisely, our construction yields only the transvection group.
\end{remark}

In~\cite{Freyn12b} we proved for OSAKAs (and hence affine Kac-Moody symmetric spaces) the following result

\begin{theorem}
\label{eithercompactornoncompact}
Every real form of a complex geometric affine Kac-Moody algebra is either of compact type or of non-compact type. A mixed type is not possible.
\end{theorem}

As a consequence of the proof of theorem~\ref{eithercompactornoncompact}, given in~\cite{Freyn12b}, it follows that a real form is of compact type iff the coefficient of the central element $c$ is imaginary and of non-compact type iff the coefficient of the central element $c$ is real. As a consequence of this observation we can note the following result:

\begin{proposition}\
\begin{itemize}
\item $H_{k,\epsilon=1, \mathbb{C}}\otimes  \mathbb{R}d$ is a non-compact real Kac-Moody algebra in the sense of theorem \ref{eithercompactornoncompact},
\item $H_{k,\epsilon=i, \mathbb{C}}\otimes i\mathbb{R}d$ is a compact real Kac-Moody algebra in the sense of theorem \ref{eithercompactornoncompact}.
\end{itemize}
\end{proposition}

\section{Affine Kac-Moody symmetric spaces of Euclidean type}
\label{section:Symmetric_spaces_Euclidean_type}

The $n$-dimensional simply connected Euclidean symmetric space is isometric to $\mathbb{R}^n$ equipped with its Euclidean metric; it is irreducible iff $n=1$.  It is non-compact and trivially diffeomorphic to a vector space. There is also a compact Euclidean symmetric space, namely the torus $T^n:=(S^1)^{n}$; it is irreducible iff its dimension $1$, hence if it is a circle $S^{1}$. As the fundamental group of $S^1$ is $\mathbb{Z}$, the fundamental group of a $n$-dimensional torus is $\mathbb{Z}^{n}$ and the torus can be described as a quotient of its universal cover $\mathbb{R}^n$ by the isometric action of its fundamental group. 

\begin{displaymath}
T^n\cong \mathbb{R}^n/ \mathbb{Z}^n \quad\textrm{and}\quad S^1\cong \mathbb{R}/\mathbb{Z}\, .
\end{displaymath}

In both of these irreducible cases the OSLA is isomorphic to $\mathfrak{g}\cong SO(1)\ltimes \mathbb{R}$, the involution is the inversion $x\mapsto -x$. The fixed subgroup is \begin{displaymath}
SO(1):=\{\pm 1\}\, .
\end{displaymath} 

Thus in the classification of simply connected irreducible Riemannian symmetric spaces only $\mathbb{R}^n$ appears. Nevertheless there is a duality construction similar to the duality between symmetric spaces of compact type and of noncompact type, relating the two spaces. With respect to this construction, the torus $T^n$ is the natural ``compact dual'' of the noncompact symmetric space $\mathbb{R}^n$.

In terms of Lie algebras and the exponential function we can describe the ``duality'' between
\begin{displaymath}
\mathbb{R}\Leftrightarrow S^1
\end{displaymath}
as follows:

Start with the space of ``compact type'', namely $S^1$. Its tangential space $T_p(S^1)$ --- the $\mathcal{P}$-component of the Cartan decomposition --- is isomorphic to $\mathbb{R}$. Its complexification is isomorphic to $\mathbb{C}$. $T_p(S^1)$ is embedded into $\mathbb{C}$ as the imaginary axis.  Dualization exchanges the subspace $i\mathbb{R}$ by the subspace $\mathbb{R}$.
The exponential mapping in the first (compact) case (writing the symmetric space $\mathbb{R}\cong \mathbb{R}^+$) is 

\begin{displaymath}
\exp: i\mathbb{R}\longrightarrow S^1,\quad ix\mapsto \exp(ix)=\cos(x)+i\sin(x)\, ,
\end{displaymath}

\noindent while it is in the second (non-compact) case
\begin{displaymath}
\exp: \mathbb{R}\longrightarrow \mathbb{R},\quad x\mapsto \exp(x)=\cosh(x)+\sinh(x)\, .
\end{displaymath}

\noindent We have furthermore the complex relation

\begin{displaymath}
\exp: \mathbb{C}\longrightarrow \mathbb{C}^*,\quad z\mapsto \exp(z)\, .
\end{displaymath}

\noindent Hence we can desribe the symmetric space of non-compact type also as 
\begin{displaymath}
M=\mathbb{C}^*/S^1\, .
\end{displaymath}
similar to the description of finite dimensional symmetric spaces of type $IV$ in the classification (see for example~\cite{Helgason01}, chapter X.
 
The theory of Euclidean affine Kac-Moody symmetric spaces is similar:

An affine geometric Kac-Moody algebra of Euclidean type is the double extension of a loop algebra $\widehat{L}(\mathfrak{g}_{\mathbb{R}})$ where $\mathfrak{g}$ denotes a finite dimensional real Abelian Lie algebra.  This real Abelian Lie algebra is a real form of a complex Lie algebra $\widehat{L}(\mathfrak{g}_{\mathbb{C}})$ with $\mathfrak{g}_{\mathbb{C}}=\mathfrak{g}_{\mathbb{R}}\oplus i \mathfrak{g}_{\mathbb{R}}$. Hence, we have two possible real forms: $\widehat{L}(\mathfrak{g}_{\mathbb{R}})$ and $\widehat{L}(i\mathfrak{g}_{\mathbb{R}})$. The loop algebras are isomorphic as are their Heisenberg algebras. Nevertheless, the associated Kac-Moody groups are different.

Let us now describe the metric on $\widehat{L}(\mathfrak{g})$. This metric is then extendet to $\widehat{L}(G)$ bei left translation.

We start with the Euclidean scalar product $\langle\cdot,\cdot \rangle$ on $\mathfrak{g}_{\mathbb{R}}$ defined as follows (see also~\cite{Heintze08}, \cite{HPTT}):
\begin{eqnarray*}
\langle u, v \rangle &= &\frac{1}{2\pi} \int\langle u(t), v(t) \rangle dt,\\
\langle c, d \rangle &= &-1,\\
\langle c, c \rangle &= &\langle d, d \rangle \hspace{5pt} =\hspace{5pt} \langle f, c \rangle \hspace{5pt} = \hspace{5pt} \langle f, d \rangle \hspace{5pt}= \hspace{5pt}0\, .
\end{eqnarray*}

\noindent This scalar product $\langle\cdot ,\cdot\rangle$ is clearly $Ad$-invariant.

Now we extend $\langle\cdot,\cdot\rangle$ to the complexification $\mathfrak{g}_{\mathbb{C}}$ of $\mathfrak{g}$:
We identify $\widehat{L}(\mathfrak{g}_{\mathbb{C}})=\widehat{L}(\mathfrak{g}_{\mathbb{R}})\oplus i \widehat{L}(\mathfrak{g}_{\mathbb{R}})$ and put $\langle ix, iy \rangle:= \langle x,y\rangle$ and $\langle x, iy\rangle=0$. This scalar product is again  $Ad$-invariant.

\begin{definition}[Kac-Moody symmetric space of Euclidean type - 1. case]
Let $\mathfrak{g}_{\mathbb{C}}:=\mathbb{C}^n$ and let $G_{\mathbb{C}}\cong (\mathbb{C}^*)^n$ be the associated Lie group. We put
\begin{displaymath}{L(G_{\mathbb{C}})}_{\mathbb{R}}:=\{f:\mathbb{C}^*\longrightarrow G_{\mathbb{C}}|f\textrm{ holomorphic}, f(S^1)\subset i\mathfrak{g}_{\mathbb{R}}\subset \mathfrak{g}_{\mathbb{C}}\}\, .\end{displaymath}
Let furthermore $\widehat{L}(G)$ denote the Kac-Moody group constructed from $L(G)$, that is the semidirect product of the Heisenberg group with $\mathbb{C}^*$. Then the space $\widehat{L}(G_{\mathbb{C}})_{\mathbb{R}}$ is a Kac-Moody symmetric space of the Euclidean type.  
\end{definition}

Spaces of this type can occur as a subspace of an indecomposable Kac-Moody symmetric space having subspaces of compact type.

\begin{definition}[Kac-Moody symmetric space of Euclidean type - 2. case]
Let $\mathfrak{g}_{\mathbb{C}}:=\mathbb{C}^n$ and let $G_{\mathbb{C}}\cong (\mathbb{C}^*)^n$ be the associated Lie group. Define furthermore
\begin{displaymath}{L(G_{\mathbb{C}})}:=\{f:\mathbb{C}^*\longrightarrow G_{\mathbb{C}}|f\textrm{ holomorphic}\}\end{displaymath}
\begin{displaymath}{L(G_{\mathbb{C}})}_{\mathbb{R}}:=\{f:\mathbb{C}^*\longrightarrow G_{\mathbb{C}}|f\textrm{ holomorphic}, f(S^1)\subset i\mathfrak{g}_{\mathbb{R}}\subset \mathfrak{g}_{\mathbb{C}}\}\end{displaymath}
Then the space $M=\widehat{L}(G_{\mathbb{C}})/\widehat{L}(G_{\mathbb{C}})_{\mathbb{R}}$ is a Kac-Moody symmetric space of the Euclidean type.  
\end{definition}

Spaces of this type can occur as a subspace of an indecomposable Kac-Moody symmetric space having subspaces of non-compact type.

\begin{lemma}
A Kac-Moody symmetric space of Euclidean type is irreducibe iff $\mathfrak{g}\cong \mathbb{C}$.
\end{lemma}

\begin{proof}
A Kac-Moody symmetric space is irreducible iff its OSAKA is irreducible (see~\cite{Freyn13b}). An OSAKA of Euclidean type is irreducible if its Kac-Moody algebra is irreducible. Hence this result is a consequence of definitions~\ref{def:Heisenberg} and \ref{def:generalized_Heisenberg} and the observation, that a Heisenberg algebra is a subalgebra of any generalized Heisenberg algebra. Thus let $H_g$ be a generalized Heisenberg algebra and $H_0$ a Heisenberg subalgebra of $H_g$. Then the extension $H_0\oplus \mathbb{C}d$ is a Kac-Moody algebra of Euclidean type, which is a subalgebra of $H_{g}\oplus \mathbb{C}d$. Hence $H_{g}\oplus \mathbb{C}d$ is reducible 
\end{proof}

\begin{lemma}
A Kac-Moody symmetric space of Euclidean type is flat.
\end{lemma}

\begin{proof}
Let $g,h\in \widehat{L}(\mathfrak{a})$. Then by definition $[g,h]\in \mathbb{R}c$. Using definition~\ref{def:sectional_curvature} we get \begin{displaymath}
\langle R(f)\{g,h,g\}, h\}\rangle =\langle[g,[h,g]],h\rangle=\langle[h,g],[h,g]\in \mathbb{R}\langle c,c\rangle= 0\, .
\end{displaymath}
Hence the curvature of the Euclidean Kac-Moody group and in consequence of the symmetric space vanishes; thus Kac-Moody symmetric spaces of Euclidean type are flat.
\end{proof}

The most important feature of Kac-Moody symmetric spaces of the Euclidean type is that their exponential maps behave well: 

\begin{proposition}
The exponential map is a tame diffeomorphism $\Lexp: \widehat{U}\longrightarrow \widehat{V}$, where $\widehat{U}\subset \widehat{L}(\mathfrak{g})$ and $\widehat{V}\subset \widehat{L}(G)$.
\end{proposition}

\noindent This is a direct consequence of corollary 4.25 in~\cite{Freyn12d}, stating that the exponential map of a loop group associated to a Lie groups whose universal cover is biholomorphically equivalent to $\mathbb{C}^n$ is a local diffeomorphism.

\noindent For commutative Fr\'echet Lie groups, Galanis~\cite{Galanis96} proves the following result:

\begin{theorem}[Galanis]
 Let $G$ be a commutative Fr\'echet Lie group modelled on a Fr\'echet space $F$. Assume the group is a strong exponential Lie group (i.e. a group such that $\exp$ is a local diffeomorphism). Then it is a projective limit Banach Lie group.
\end{theorem}

As the loop group subjacent to a Euclidean Kac-Moody group is Abelian and strong exponential and the $2$-dimensional extension does not cause any functional analytic problem, we can apply this result to get immediately that an Euclidean Kac-Moody symmetric space is strong exponential. Furthermore it carries a projective limit Banach structure modelled via the exponential maps on the tangential Kac-Moody algebra. Remark, that one can prove this also elementary by calculation in this setting; for the general investigation of inverse limit structures on Kac-Moody symmetric spaces see~\cite{Freyn09} and the forthcoming paper~\cite{Freyn12g}.

\begin{theorem}
Let $G$ be a simple compact Lie group, $T\subset G$ a maximal torus. The space $\widehat{MT}$ is a Kac-Moody symmetric space of Euclidean type. 
\end{theorem}

As is the case of finite dimensional symmetric spaces of the Euclidean type the isometry group is much bigger, namely a semidirect product of a Euclidean Kac-Moody group with the isotropy group of a point.

\bibliographystyle{alpha}
\bibliography{Doktorarbeit1}

\end{document}